\documentclass[12pt,reqno]{amsart}
\usepackage{amssymb,amsmath,amsthm}
\usepackage{a4}
\usepackage{enumitem}
\usepackage{color}
\usepackage[colorlinks=true,
linkcolor=webgreen,
filecolor=webbrown,
citecolor=webgreen]{hyperref}
\definecolor{webgreen}{rgb}{0,0,1}
\definecolor{recrown}{rgb}{1,.2,.6}
\usepackage{fullpage}
\usepackage{orcidlink}

\begin{document}
\newtheorem{theorem}{Theorem}
\newtheorem{corollary}[theorem]{Corollary}
\newtheorem{lemma}[theorem]{Lemma}
\theoremstyle{definition}
\newtheorem{example}{Example}
\newtheorem*{examples}{Examples}
\newtheorem*{notation}{Notation}
\theoremstyle{theorem}
\newtheorem{thmx}{Theorem}
\renewcommand{\thethmx}{\text{\Alph{thmx}}}
\newtheorem{lemmax}{Lemma}
\renewcommand{\thelemmax}{\text{\Alph{lemmax}}}
\hoffset=-0cm
\theoremstyle{definition}
\newtheorem*{definition}{Definition}
\newtheorem*{remark}{\bf Remark}
\title{\bf Some factorization results for bivariate polynomials}
\author{Nicolae Ciprian Bonciocat$^1$}
\author{Rishu Garg$^{2}$ {\large \orcidlink{0009-0008-2348-8340}}}
\author{Jitender Singh$^{3,\dagger}$ {\large \orcidlink{0000-0003-3706-8239}}}
\address[1]{Simion Stoilow Institute of Mathematics of the Romanian Academy, Research Unit 7, P.O. Box 1-764, Bucharest 014700, Romania\newline
{\tt nicolae.bonciocat@imar.ro}}
\address[2,3]{Department of Mathematics,
Guru Nanak Dev University, Amritsar-143005, India\newline 
{\tt rishugarg128@gmail.com}, {\tt jitender.math@gndu.ac.in}}
\markright{}
\date{}
\footnotetext[2]{


2020MSC: {Primary 11R09; 12J25; 12E05; 11C08}\\

}
\maketitle
\newcommand{\K}{\mathbb{K}}
\begin{abstract}
We provide upper bounds on the total number of irreducible factors,
and in particular irreducibility criteria for some classes of bivariate polynomials $f(x,y)$ over an arbitrary field $\K$. Our results rely on information on the degrees of the coefficients of $f$, and on information on the factorization of the constant term and of the leading coefficient of $f$, viewed as a polynomial in $y$ with coefficients in $\K[x]$. In particular, we provide a generalization of the bivariate version of Perron's irreducibility criterion, and similar results for polynomials in an arbitrary number of indeterminates. The proofs use non-Archimedean absolute values, that are suitable for finding information on the location of the roots of $f$ in an algebraic closure of $\K(x)$.
\end{abstract}
\section{Introduction}
 Many of the classical irreducibility results for univariate integer polynomials $f(x)$ rely on information on the canonical decomposition of their coefficients, such as the famous irreducibility criteria of Sch\"onemann \cite{Schonemann}, Eisenstein \cite{Eisenstein}, and Dumas \cite{Dumas}, or on information on the canonical decomposition of the values that they assume at some suitable integer arguments, such as the irreducibility criteria of St\"ackel \cite{Stackel}, Dorwart \cite{Dorwart}, Ore \cite{Ore}, Weisner \cite{Weisner}, and Cohn \cite{PolyaSzego}. Other results, such as P\'olya's irreducibility criterion \cite{Polya} rely only on the magnitude of $f(m_1),\dots ,f(m_n)$, with disregard to their canonical decomposition, for distinct integers $m_1,\dots ,m_n$, or on the comparative size of the coefficients, as in the irreducibility criterion of Perron \cite{Perron}. For a quick review of some classical or more recent such results the reader is referred to \cite{Guersenzvaig1} and \cite{SKJS2023}, for instance.
In more recent times, many fundamental results and new ideas appeared on this topic, for instance in the works of Brillhart, Filaseta and Odlyzko \cite{Brillhart}, Cole, Dunn, Filaseta and Gross \cite{CDF}, \cite{Filaseta2}, \cite{FilasetaGross}, Murty \cite{RamMurty}, Girstmair \cite{Girstmair} (see also \cite{J-S-2}), Guersenzvaig \cite{Guersenzvaig1}, Filaseta and Luckner \cite{FilasetaLuckner}, and Weintraub \cite{Weintraub}, to name just a few. Other recent criteria are inspired by some results on Hilbert Irreducibility Theorem, and use some properties of the resultant of two polynomials, combined with information on the magnitude of their coefficients. Here we refer the reader to some elegant methods developed by Cavachi, V\^aj\^aitu and Zaharescu in \cite{CAV}, \cite{CVZ1} and \cite{CVZ2}, which are useful in the study of linear combinations of relatively prime polynomials, and which have been adapted by various authors for the study of compositions and multiplicative convolutions of polynomials.

 Finding information on the location of the roots is one of the main ingredients in the proofs of many of these results. Most of these criteria admit natural generalizations to the multivariate case. Given a bivariate polynomial $f(x,y)$ over an arbitrary field $\K$, if we regard $f$ as a polynomial in $y$ with coefficients in $\K[x]$, a natural way to study its factorization is to try to adapt the methods available in the univariate case. One of the most suitable ways to study the location of the roots of $f$ in an algebraic closure of $K(x)$ is to use non-Archimedean absolute values, and to employ the ideas and the techniques in the univariate case, where the usual absolute value is used. Such an approach was used for instance in \cite{NBZ2008}, \cite{BBCM2012}, \cite{BBBC2022} and \cite{BBBCM2023} to obtain multivariate versions of the irreducibility criteria of Cohn and P\'olya, and to derive irreducibility results and bounds on the number of irreducible factors counted with multiplicities for some classes of multivariate polynomials.

 The irreducibility criterion of Perron \cite{Perron} for integer univariate polynomials was recently generalized in \cite{JRG2023} and \cite{JRG2}, and has the following natural generalization for bivariate polynomials over an arbitrary field. \begin{thmx}[Perron \cite{Perron}]\label{th:B}
Let $\K$ be a field and $f=a_0(x)+\cdots+a_{n-1}(x)y^{n-1}+a_n(x)y^n\in \K[x,y]$ with $n\geq 2$, $a_0,\ldots,a_{n-1}\in \K[x]$, $a_n\in \K$, and $a_0a_n\neq 0$. If
\begin{eqnarray*}
\deg a_{n-1}>\max\thinspace \{\deg a_0,\deg a_1,\dots,\deg a_{n-2}\},
\end{eqnarray*}
then $f$ is irreducible over $\K[x]$.
\end{thmx}
 For a proof of Theorem \ref{th:B} that relies on the use of non-Archimedean absolute values, and its extension to an arbitrary number of indeterminates, we refer the reader to \cite{Bonciocat2}.

 In this paper, we will first prove several factorization results for some classes of bivariate polynomials $f(x,y)=\sum a_i(x)y^i$ over a given field $\K$, that provide upper bounds for the number of irreducible factors of $f$ over $\K[x]$, counted with their multiplicities. These bounds depend only on the factorization of $a_0$ and $a_n$, and hold for polynomials for which the degrees of $a_0,\dots ,a_n$ satisfy certain inequalities.
In particular, these results provide irreducibility criteria for bivariate polynomials for which $a_0$ or $a_n$ is irreducible over $\K$. We will then prove some extensions of Theorem \ref{th:B} to the case when a coefficient other than $a_{n-1}$ has degree greater than the degrees of all the other coefficients. Several similar results will be finally obtained for polynomials in an arbitrary number of indeterminates. Our proofs will use non-Archimedean absolute values, and some explanations will rely on arguments coming from Newton Polygon theory, via Dumas' Theorem \cite{Dumas}. We mention here that the bounds on the number of irreducible factors that we will obtain are best possible, and examples in this respect will be provided.

 The paper is organized as follows. The main results are stated in Section \ref{sec:2}, and their proofs are given in Section \ref{sec:3}. In the last section of the paper we provide several examples of infinite families of irreducible polynomials, as well as families of polynomials for which bounds on their total number of irreducible factors can be deduced from our results.
\section{Main results}\label{sec:2}
 Our first results provide upper bounds on  the total number of factors for some classes of bivariate polynomials $f(x,y)=a_0(x)+a_1(x)y+\cdots +a_n(x)y^n\in \K[x,y]$, using information on the degrees of $a_i$ combined with information on the factorization of $a_0$ and $a_n$. Even if they follow by a similar argument, for the sake of symmetry we will also include in our statements the results for the reciprocal of $f$ with respect to $y$. Such results are often overlooked, even if they usually provide useful information on the factorization of $f$, and as we shall later see, by simultaneously studying $f$ and its reciprocal, one may obtain sharper results.

\begin{theorem}\label{th:0}
Let $\K$ be a field and $f=a_0(x)+a_1(x)y+\cdots+a_n(x) y^n\in \K[x,y]$, with $n\geq 2$, $a_0,\ldots, a_n\in \K[x]$, $a_0a_n\neq 0$. Assume that $f$ has no nonconstant factor in $\K[x]$, and let $\nu_0$ and $\nu_n$ be the number of irreducible factors of $a_0(x)$ and $a_n(x)$ in $\K[x]$, respectively, counted with their multiplicities.

i) If $\deg a_0>\max \thinspace \{\deg a_1,\dots ,\deg a_n\}$, then $f$ is a product of at most $\nu_0$ irreducible polynomials over $\K[x]$.

ii) If $\deg a_n>\max \thinspace \{\deg a_0,\dots ,\deg a_{n-1}\}$, then $f$ is a product of at most $\nu_n$ irreducible polynomials over $\K[x]$.
\end{theorem}

In each one of the two cases in Theorem \ref{th:0} we may obtain a potentially sharper result, provided some information on the factorization of both $a_0$ and $a_n$ is available, as in the following result.
\begin{theorem}\label{th:1}
Let $\K$ be a field and $f=a_0(x)+a_1(x)y+\cdots+a_n(x) y^n\in \K[x,y]$, with $n\geq 2$, $a_0,\ldots, a_n\in \K[x]$, $a_0a_n\neq 0$. Assume that $f$ has no nonconstant factor in $\K[x]$, and let $\nu_0$ and $\nu_n$ be the number of irreducible factors of $a_0(x)$ and $a_n(x)$ in $\K[x]$, respectively, counted with their multiplicities. Then $f$ is a product of at most $\nu=\min\{\nu_0,\nu_n\}$ irreducible polynomials over $\K[x]$ in each one of the following two cases:

i) $a_0$ is reducible over $\K$, $\deg a_n\geq \deg a_0-\deg q$, where $q\in \K[x]$ is an irreducible factor of $a_0$ of smallest degree, and $\deg a_0>\max \thinspace \{\deg a_1,\dots ,\deg a_n\}$;

ii) $a_n$ is reducible over $\K$, $\deg a_0\geq \deg a_n-\deg q$, where $q\in \K[x]$ is an irreducible factor of $a_n$ of smallest degree, and $\deg a_n>\max \thinspace \{\deg a_0,\dots ,\deg a_{n-1}\}$.
\end{theorem}
We mention that the upper bounds on the number of irreducible factors of $f$ in Theorem \ref{th:0} and Theorem \ref{th:1} are best possible, in the sense that there exist examples of infinite families of bivariate polynomials for which these bounds are attained. To see this, let for example $p$ and $q$ be prime numbers, $n$ a positive integer, and let
\begin{eqnarray*}
f_1(x,y) = pq+(p+q)x^n+x^{2n}+(p+q+2x^n)y^n+y^{2n} \in \mathbb{Q}[x,y].
\end{eqnarray*}
One may check that $f_1$ satisfies the hypothesis of Theorem \ref{th:0} i) with $a_0=pq+(p+q)x^n+x^{2n}$, which is the product of two irreducible factors in $\mathbb{Q}[x]$, namely $p+x^n$ and $q+x^n$, the first one being Eisensteinian with respect to $p$, and the second one with respect to $q$. By Theorem \ref{th:0}, we conclude that our polynomial $f_1(x,y)$ is the product of at most two irreducible factors over $\mathbb{Q}[x]$. On the other hand, we observe that actually $f_1$ decomposes as
\[
f_1(x,y)=(p+x^n+y^n)(q+x^n+y^n),
\]
which is the product of two irreducible factors over $\mathbb{Q}[x]$, as $p+x^n+y^n$ and $q+x^n+y^n$ are Eisensteinian too, the first one with respect to the irreducible polynomial $p+x^n$, and the second one with respect to the irreducible polynomial $q+x^n$. For an example related to Theorem \ref{th:0} ii), one may for instance consider the polynomial
\begin{eqnarray*}
f_2(x,y) = 1+(p+q+2x^n)y^n+(pq+(p+q)x^n+x^{2n})y^{2n}\in \mathbb{Q}[x,y],
\end{eqnarray*}
which is the reciprocal of $f_1$ with respect to $y$, and whose irreducible factors are precisely the reciprocals with respect to $y$ of the irreducible factors of $f_1$.

 Similarly, the bound in Theorem \ref{th:1} is best possible too, since it is attained by all the polynomials
\begin{eqnarray*}
f_3(x,y) = (r+x^2)^4+2(p+x^3)(r+x^2)^2y^n+(p+x^3)^2y^{2n}\in \mathbb{Q}[x,y],
\end{eqnarray*}
where $p$ and $r$ are prime numbers, and $n$ is a positive integer. Here, $a_0=(r+x^2)^4$ with $q(x)=r+x^2$, which is Eisensteinian with respect to $r$, and $a_{2n}=(p+x^3)^2$, with $p+x^3$ Eisensteinian with respect to $p$. Thus $\nu_0=4$, $\nu_{2n}=2$, $\deg a_{2n}=\deg a_0-\deg q=6$, and $\deg a_0=8>7=\max\{\deg a_{n},\deg a_{2n}\}$. By Theorem \ref{th:1} i), we conclude that $f_3$ is a product of at most $2=\min\{\nu_0,\nu_{2n}\}$ irreducible factors over $\mathbb{Q}[x]$. On the other hand, we actually have $f_3(x,y)=h(x,y)^2$, with $h(x,y)=(r+x^2)^2+(p+x^3)y^n$, which is irreducible over $\mathbb{Q}[x]$, being the reciprocal with respect to $y$ of the polynomial $(p+x^3)+(r+x^2)^2y^n$, which is Eisensteinian with respect to $p+x^3$. For an example related to Theorem \ref{th:1} ii), one may obviously choose the reciprocal of $f_3$ with respect to $y$.

 The following irreducibility criteria are immediate applications of Theorems \ref{th:0} and \ref{th:1}.
\begin{corollary}\label{c2}
Let $\K$ be a field. Let $f=a_0(x)+a_1(x)y+\cdots+a_n(x) y^n\in \K[x,y]$, with $n\geq 2$, $a_0,\ldots, a_n\in \K[x]$, $a_0a_n\neq 0$, be such that $f$ has no nonconstant factor in $\K[x]$ and
\begin{eqnarray*}
\deg a_0>\max\thinspace \{\deg a_1, \ldots, \deg a_n \}.
\end{eqnarray*}
If $a_0$ is irreducible in $\K[x]$, or if $a_n$ is irreducible in $\K[x]$ and $\deg a_n\geq \deg a_0-\deg q$, with $q\in \K[x]$ an irreducible factor of $a_0$ of smallest degree, then $f$ is irreducible over $\K[x]$.
\end{corollary}
\begin{corollary}\label{c2reciprocal}
Let $\K$ be a field. Let $f=a_0(x)+a_1(x)y+\cdots+a_n(x) y^n\in \K[x,y]$, with $n\geq 2$, $a_0,\ldots, a_n\in \K[x]$, $a_0a_n\neq 0$, be such that $f$ has no nonconstant factor in $\K[x]$ and
\begin{eqnarray*}
\deg a_n >\max \thinspace \{\deg a_0, \ldots, \deg a_{n-1} \}.
\end{eqnarray*}
If $a_n$ is irreducible in $\K[x]$, or if $a_0$ is irreducible in $\K[x]$ and $\deg a_0\geq \deg a_n-\deg q$, where $q\in \K[x]$ is an irreducible factor of $a_n$ of smallest degree, then $f$ is irreducible over $\K[x]$.
\end{corollary}
 We mention here that Theorem \ref{th:0} and Theorem \ref{th:1} extend  some of the results in \cite{J-S-4} and \cite{Bevelacqua} to the bivariate case and that the conclusion in Corollary \ref{c2} on the irreducibility of $f(x,y)$ over $\K[x]$ when $a_0$ is irreducible in $\K[x]$ also follows by Lemma 1 in \cite{NBZ2008}.

For the case when a coefficient of $f$ other than the leading one has degree greater than the degrees of all the other coefficients, we have the following factorization result, which generalizes the multivariate version of Perron's irreducibility criterion \cite{Perron}, \cite{Bonciocat2}.
\begin{theorem}\label{th:2}
Let $\K$ be a field, and let $f=a_0(x)+a_1(x)y+\cdots+a_n(x) y^n\in \K[x,y]$ with $n\geq 2$, $a_0,\ldots, a_{n-1}\in \K[x]$, $a_n\in \K$, $a_0 a_n\neq 0$, be such that there exists an index $j$ with $0\leq j\leq n-1$ for which
\begin{eqnarray*}
\deg a_j >\max_{i\neq j}\thinspace \deg a_i.
\end{eqnarray*}
Then $f$ is a product of at most $n-j$ irreducible polynomials over $\K[x]$. In particular, if $j=n-1$, then $f$ is irreducible over $\K[x]$.
\end{theorem}
 One may improve the bound on the number of irreducible factors in Theorem \ref{th:2} for the case when $a_0$ too belongs to $\K$.
\begin{theorem}\label{th:2withreciprocal}
Let $\K$ be a field, and let $f=a_0(x)+a_1(x)y+\cdots+a_n(x) y^n\in \K[x,y]$ with $n\geq 2$, $a_1,\ldots, a_{n-1}\in \K[x]$, $a_0,a_n\in \K$, $a_0 a_n\neq 0$, be such that there exists an index $j$ with $1\leq j\leq n-1$ for which
\begin{eqnarray*}
\deg a_j>\max_{i\neq j}\thinspace \deg a_i.
\end{eqnarray*}
Then $f$ is a product of at most $\min\{j,n-j\}$ irreducible polynomials over $\K[x]$. In particular, if $j=1$ or $j=n-1$, then $f$ is irreducible over $\K[x]$.
\end{theorem}

 The upper bound $\min\{j,n-j\}$ in the statement of Theorem \ref{th:2withreciprocal} is also best possible, in the sense that there exist polynomials for which this bound is attained. To see this, consider the polynomial $f_4(x,y)=(1+xy+y^2)^j\in\mathbb{Q}[x,y]$ with $j$ a positive integer. We have $\deg _yf_4=2j$ and $\deg _xf_4=j$, and by using the multinomial rule, if we write $f_4$ as a polynomial in $y$ with coefficients in $\mathbb{Q}[x]$, say $f_4=a_0(x)+a_1(x)y+\cdots +a_{2j}(x)y^{2j}$, we see that $f_4$ has a unique coefficient of maximal degree, which is $a_j$, and $\deg a_j=j$. Since $a_0=a_{2j}=1$, we deduce by Theorem \ref{th:2withreciprocal} that $f_4$ is the product of at most $j$ irreducible factors. Indeed, this is the case, as one may easily check that $1+xy+y^2$ is irreducible over $\mathbb{Q}[x]$. A more general related example with a polynomial of arbitrary degree $n$ and $j$ a divisor of $n$ will be given in the last section of the paper.

 Interestingly, Theorem \ref{th:2} extends to arbitrary polynomials, with no need to asking $a_n$ to be a unit in $\K[x]$, as in the following result.
 \begin{theorem}\label{th:3}
Let $\K$ be a field and $f=a_0(x)+a_1(x)y+\cdots+a_n(x) y^n\in \K[x,y]$ with $n\geq 2$, $a_0,\ldots, a_n\in \K[x]$,  $a_0 a_n\neq 0$. Assume that $f$ has no nonconstant factor in $\K[x]$. If  there exists an index $j$ with $0\leq j\leq n-1$ for which
\begin{eqnarray*}
\deg a_j>\max_{i\neq j}\thinspace \{\deg a_i+(j-i)\deg a_n\},
\end{eqnarray*}
then $f$ is a product of at most $n-j$ irreducible polynomials over $\K[x]$. In particular, if $j=n-1$, then $f$ is irreducible over $\K[x]$.
\end{theorem}
 The bound $n-j$ on the number of irreducible factors in Theorem \ref{th:3} can be also improved to $\min\{j,n-j\}$, provided we also consider the reciprocal of $f$ with respect to $y$, and impose a stronger condition on the degree of $a_j$, as follows.
\begin{theorem}\label{th:3withreciprocal}
Let $\K$ be a field and $f=a_0(x)+a_1(x)y+\cdots+a_n(x) y^n\in \K[x,y]$ with $n\geq 2$, $a_0,\ldots, a_n\in \K[x]$,  $a_0 a_n\neq 0$. Assume that $f$ has no nonconstant factor in $\K[x]$. If there exists an index $j$ with $1\leq j\leq n-1$ for which
\begin{eqnarray*}
\deg a_j & > & \max_{i<j}\thinspace \{\deg a_i+(j-i)\deg a_n\} \quad {\rm and}\\
\deg a_j & > & \max_{i>j}\thinspace \{\deg a_i+(i-j)\deg a_0\},
\end{eqnarray*}
then $f$ is a product of at most $\min\{j,n-j\}$ irreducible polynomials over $\K[x]$. In particular, if $j=1$ or $j=n-1$, then $f$ is irreducible over $\K[x]$.
\end{theorem}
We notice here that Theorem \ref{th:2withreciprocal} can be recovered from Theorem \ref{th:3withreciprocal} as the special case when both the coefficients $a_0$ and $a_n$ belong to $\K$, since in that case the two inequalities above simply reduce to $\deg a_j > \max_{i\neq j}\thinspace \deg a_i$.

The results stated so far extend easily to polynomials in an arbitrary number of indeterminates. For any positive integer $r\geq 2$ and any polynomial $f\in \K[x_1,\ldots,x_r]$, let $\deg_r f$ be the degree of $f$ viewed as a polynomial in the variable $x_r$ and coefficients in $\K[x_1,\ldots,x_{r-1}]$. We will only mention here two such results, each of which is obtained by induction on $r$. For instance, as an immediate consequence of Theorem \ref{th:2}, we have the following factorization result.
\begin{theorem}\label{th:4}
Let $\K$ be a field and let $s\geq 3$ be a positive integer. Let $f=\sum_{i=0}^n a_ix_s^i\in \K[x_1,\ldots, x_s]$ with $n\geq 2$, $a_0,\ldots, a_{n-1}\in \K[x_1,\ldots, x_{s-1}]$, $a_n\in \K[x_1,\ldots, x_{s-2}]$, $a_0 a_n\neq 0$, and assume that there exists an index $j$ with $0\leq j\leq n-1$ for which
\begin{eqnarray*}
\deg_{s-1} a_j>\max_{i\neq j}\thinspace \deg_{s-1} a_i.
\end{eqnarray*}
Then $f$ is a product of at most $n-j$ irreducible polynomials over $\K[x_1,\ldots, x_{s-1}]$. In particular, if $j=n-1$, then $f$ is irreducible over $\K[x_1,\ldots, x_{s-1}]$.
\end{theorem}
 We mention that Theorem \ref{th:4} generalizes Theorem C in \cite{Bonciocat2}, which is the irreducibility criterion that corresponds to $j=n-1$ in the above result.
Similarly, Theorem  \ref{th:3} gives the following  result for multivariate polynomials.
 \begin{theorem}\label{th:5}
 Let $\K$ be a field and let $s\geq 3$ be a positive integer. Let $f=\sum_{i=0}^n a_ix_s^i\in \K[x_1,\ldots, x_s]$ with $n\geq 2$, $a_0,\ldots, a_{n}\in \K[x_1,\ldots, x_{s-1}]$,  $a_0 a_n\neq 0$. Assume that $f$ has no nonconstant factor in $\K[x_1,\ldots, x_{s-1}]$. If  there exists an index $j$ with $0\leq j\leq n-1$ for which
\begin{eqnarray*}
\deg_{s-1} a_j>\max_{i\neq j}\thinspace \{\deg_{s-1} a_i+(j-i)\deg_{s-1} a_n\},
\end{eqnarray*}
then $f$ is a product of at most $n-j$ irreducible polynomials over $\K[x_1,\ldots, x_{s-1}]$. In particular, if $j=n-1$, then $f$ is irreducible over $\K[x_1,\ldots, x_{s-1}]$.
\end{theorem}
\section{Proofs of the main results}\label{sec:3}
 Our factorization results for polynomials over $\K[x]$ will be proved in a non-Archimedean setting, and by Gauss's Lemma, they will imply similar factorization results over $\K(x)$.

\begin{proof}[\bf Proof of Theorem \ref{th:0}]
i) We will first recall a well-known non-Archimedean absolute value, which will be used in our proofs. Given a field $\K$, let $\K(x)$ denote the field of fractions of the polynomial ring $\K[x]$. Let $\rho>1$ be a real number, define $\deg(0)=-\infty$, and for all $a, b\in \K[x]$ with $b\neq 0$, define
\begin{eqnarray*}
\|a\|_\rho=\rho^{\deg(a)}\quad {\rm and}\quad \|a/b\|_\rho=\rho^{\deg(a)-\deg(b)}=\|a\|_\rho/\|b\|_\rho.
\end{eqnarray*}
In particular, we have $\|a\|_\rho\geq 1$ for all nonzero $a\in\K[x]$.
We note that $\|\cdot\|_\rho$ is a nonnegative function, and that for all $a,b\in \K[x]$ we have $\deg (a)\leq \deg(b)$ if and only if $\|a\|_\rho\leq \|b\|_\rho$. Moreover, for all $a,b\in \K(x)$ we have
\begin{eqnarray*}
\|ab\|_\rho=\|a\|_\rho\|b\|_\rho \quad {\rm and}\quad \|a+b\|_\rho\leq \max\{\|a\|_\rho,\|b\|_\rho\},
\end{eqnarray*}
so $\|\cdot\|_\rho$ is a non-Archimedean absolute value on $\K(x)$. We fix now an algebraic closure $\overline{\K(x)}$ of $\K(x)$, which is unique up to isomorphism. We also fix an extension of this absolute value from $\K(x)$ to $\overline{\K(x)}$ (see \cite{EnglerPrestel}, for instance), which will also be denoted by $\|\cdot\|_\rho$.

Let us assume that $\theta_1,\dots ,\theta _n\in \overline{\K(x)}$ are the roots of $f$. We will next prove that our assumption on the degree of $a_0$ forces $\theta_1,\dots ,\theta _n$ to satisfy the inequalities $\|\theta_i\|_\rho>1$ for each $i$.
Indeed, the inequality $\deg a_0 >\max\thinspace \{\deg a_1, \dots, \deg a_n\}$ reads $\|a_0\|_\rho >\max_{i\geq 1}\|a_i\|_\rho$, so if we assume on the contrary that $\|\theta_i\|_\rho\leq 1$ for some $i$, then since $f(\theta_i)=0$, we have $a_0(x)=-a_1(x)\theta_i-\cdots-a_n(x)\theta_i^n$, which further implies that
\begin{eqnarray*}
\|a_0\|_\rho & = & \|-a_1\theta_i-\cdots-a_n\theta_i^n\|_\rho\\
 & \leq  & \max\thinspace \{\|a_1\theta_i\|_\rho,\ldots, \|a_n\theta_i^n\|_{\rho}\}\\
 & =  & \max\thinspace \{\|a_1\|_\rho \|\theta_i\|_\rho,\ldots, \|a_n\|_\rho \|\theta_i\|_\rho^n\}\\
 & \leq   & \max\thinspace \{\|a_1\|_\rho,\ldots, \|a_n\|_\rho\},
\end{eqnarray*}
a contradiction. Thus $\|\theta_i\|_\rho>1$ for each $i=1,\dots ,n$, as claimed.

Assume now that $f$ decomposes as
\[
f(x,y)=f_1(x,y)f_2(x,y)\cdots f_r(x,y),
\]
with $f_i\in \K[x][y]$ irreducible over $\K[x]$, and hence with $\deg_y(f_i)\geq 1$ for each $i=1,\ldots,r$, where we regard $f_i$ as a polynomial in $y$ with coefficients in $\K[x]$ and denote by $\alpha_i(x)\in \K[x]$ its leading coefficient. Then for each $i=1,\ldots,r$ we have $\alpha_i\neq 0$ and
\begin{eqnarray*}
\|f_i(x,0)/\alpha_i\|_\rho=\bigl\|\prod_\theta \theta\bigr\|_\rho=\prod_\theta \|\theta\|_\rho>1,
\end{eqnarray*}
where the product is over all zeros $\theta $ of $f_i$ (which are also zeros of $f$, thus satisfying $\|\theta\|_\rho>1$). This shows that $\rho^{\deg f_i(x,0)}>\rho^{\deg \alpha_i(x)}$ for each $i$, so $\deg f_i(x,0)>\deg \alpha_i(x)$ for each $i$, implying that none of the polynomials $f_i(x,0)$ is constant. Next, since
\[
a_0(x)=f(x,0)=f_1(x,0)\cdots f_r(x,0),
\]
and this is a decomposition of $a_0$ into $r$ nonconstant factors (not necessarily irreducible over $\K$), we conclude that $r$ can not exceed $\nu_0$, so $f$ is a product of at most $\nu_0$ irreducible polynomials over $\K[x]$.

ii) We argue in exactly the same way for
\[
\tilde{f}(x,y)=y^nf(x,1/y)=a_n(x)+a_{n-1}(x)y+\cdots +a_1(x)y^{n-1}+a_0(x)y^n,
\]
the reciprocal of $f$ with respect to $y$, which has the same number of irreducible factors as $f$ (counted with their multiplicities), and whose roots in $\overline{\K(x)}$ are precisely the inverses of the roots of $f$. We may also argue directly by observing that the inequality $\deg a_n >\max\thinspace \{\deg a_0, \dots, \deg a_{n-1}\}$ reads $\|a_n\|_\rho >\max_{i<n}\|a_i\|_\rho$, and forces the roots $\theta_1,\dots ,\theta_n$ to satisfy the inequalities $\|\theta_i\|_\rho< 1$ for each $i$. Indeed, if $f$ would have a root $\theta_i$ with $\|\theta_i\|_\rho\geq 1$, then since $f(\theta_i)/\theta_i^n=0$, we would have $a_n(x)=-\frac{a_{n-1}(x)}{\theta_i}-\cdots-\frac{a_0(x)}{\theta_i^n}$, further implying that
\begin{eqnarray*}
\|a_n\|_\rho & = & \|-a_{n-1}(x)/\theta_i -\cdots-a_0(x)/\theta_i^n\|_\rho\\
 & \leq  & \max\thinspace \{\|a_{n-1}(x)/\theta_i \|_\rho,\dots ,\|a_0(x)/\theta_i^n\|_{\rho}\}\\
 & =  & \max\thinspace \{\|a_{n-1}\|_\rho /\|\theta_i\|_\rho,\dots, \|a_0\|_\rho /\|\theta_i\|_\rho^n\}\\
 & \leq   & \max\thinspace \{\|a_{n-1}\|_\rho,\dots, \|a_0\|_\rho\},
\end{eqnarray*}
a contradiction. With the same notations, this time we obtain 
$\deg f_i(x,0)<\deg \alpha_i(x)$ for each $i$, implying that none of the polynomials $\alpha_i(x)$ is constant, so the equality $a_n=\alpha_1\cdots \alpha_r$ shows that $f$ is a product of at most $\nu_n$ irreducible polynomials over $\K[x]$.
\end{proof}
\begin{proof}[\bf Proof of Theorem \ref{th:1}]
i) We already proved that all the roots $\theta $ of $f$ satisfy $\|\theta\|_\rho>1$ if $\deg a_0>\max\thinspace \{\deg a_1, \dots, \deg a_n\}$, or satisfy $\|\theta\|_\rho<1$ if $\deg a_n >\max\thinspace \{\deg a_0, \dots, \deg a_{n-1}\}$.

Let us first note that the existence of $q$ makes sense, since $a_0$ is a nonconstant polynomial. Now let us suppose again that
\[
f(x,y)=f_1(x,y)f_2(x,y)\cdots f_r(x,y)
\]
with $f_i\in \K[x][y]$ irreducible over $\K[x]$ and hence with $\deg_y(f_i)\geq 1$ for each $i=1,\ldots,r$, and denote again the leading coefficient of $f_i$ by $\alpha _i$. Since $\deg a_n\geq \deg a_0-\deg q\geq \deg q>0$, it follows that $a_n$ is a nonconstant polynomial too, and so, $\nu_n\geq 1$. We may therefore assume that $r\geq 2$, since the conclusion that $r\leq \min \{\nu_0,\nu_n\}$ holds trivially for $r=1$.

 By Theorem \ref{th:0}, we already have $r\leq \nu_0$.
On the other hand, if we fix an index $i\in\{1,\dots ,r\}$ and use the fact that $f_1(x,0)\cdots f_r(x,0)=a_0$ along with the equality $\alpha_1\cdots \alpha_r=a_n$,
we successively deduce that
\begin{eqnarray*}
\frac{\|a_0\|_\rho}{\|q\|_\rho } & \leq & \|a_n\|_\rho
<\|a_n\|_\rho \prod_{j=1,~j\neq i}^r\frac{\|f_j(x,0)\|_\rho}{\|\alpha_j\|_\rho}\\
&=&\|a_n\|_\rho  \frac{\|f_1(x,0)\cdots f_r(x,0)\|_\rho \|\alpha_i\|_{\rho }}{\|f_i(x,0)\|_{\rho } \|\alpha_1\cdots\alpha_r\|_\rho}
=\frac{\|a_0\|_\rho\|\alpha_i\|_\rho }{\|f_i(x,0)\|_\rho},
\end{eqnarray*}
which yields $\|\alpha_i\|_\rho>\|f_i(x,0)\|_\rho/\|q\|_\rho$. With the same reasoning as in the proof of Theorem \ref{th:0}, we deduce that $\deg f_i(x,0)>\deg \alpha_i(x)$, so $f_i(x,0)$ must be a nonconstant factor of $a_0$, and hence we must have $\deg f_i(x,0)\geq \deg q$, or equivalently, $\|f_i(x,0)\|_\rho\geq \|q(x)\|_\rho$. In view of this, we conclude that
\[
\|\alpha_i\|_\rho>1,
\]
uniformly for $i=1,\ldots,r$.  Thus each one of the polynomials $\alpha _1,\dots ,\alpha _r$ must have degree at least $1$, which shows that in our decomposition of $a_n$ as a product of $r$ nonconstant factors
\begin{equation*}
 a_n(x)=\alpha_1(x)\cdots\alpha_r(x),
\end{equation*}
$r$ can not exceed $\nu_n$. This completes the proof in the first case.

 For the proof of ii) we apply i) to the polynomial $\tilde{f}(x,y)=y^nf(x,1/y)$, the reciprocal of $f$ with respect to $y$.
\end{proof}

 To prove Theorem \ref{th:2}, we will need an extension to the case of multivariate polynomials of the following Theorem of Mayer \cite{Mayer} on the location of the zeros of a polynomial with real or complex coefficients, in which
one of the coefficients has absolute value greater than the sum of the absolute values of all the other coefficients.
\begin{thmx}[Mayer \cite{Mayer}]\label{th:A}
Let $\K\in \{\mathbb{R},\mathbb{C}\}$, and let $|\cdot|$ denote the Euclidean norm of $\K$.
Let $f=a_0+a_1y+\cdots+a_n y^n\in \K[y]$ be a nonconstant polynomial such that there exists an index $j$ with $0\leq j\leq n$ for which
$|a_j|>\sum_{i=0,~i\neq j}^n |a_i|$.
Then $f$ has $j$ zeros in the disk $|z|<1$, and $n-j$ zeros outside the disk $|z|\leq 1$ in the complex plane.
\end{thmx}
 The elegance of Mayer's proof of Theorem \ref{th:A} arises from the fact that it relies only on the basic notions of probability theory, the existence and uniqueness theorem for initial value problems in linear difference equations, and the Euclidean absolute value of the coefficients, which extend naturally to the non-Archimedean setting. Before stating the corresponding extension, we need to specify a topology on $\overline{\K(x)}$. The absolute value $\|\cdot\|_\rho$  induces a topology on $\overline{\K(x)}$ via the metric defined by taking the distance between $a$ and $b$ to be $\|a-b\|_\rho$
for all $a,b\in \overline{\K(x)}$. In view of this, for any $\varepsilon>0$, it is meaningful to define open and closed disks of radius $\varepsilon$ in $\overline{\K(x)}$ denoted by $\|y\|_\rho<\varepsilon$ and $\|y\|_\rho\leq \varepsilon$, and defined as the sets $\{a\in \overline{\K(x)}~|~\|a\|_\rho<\varepsilon\}$ and $\{a\in \overline{\K(x)}~|~\|a\|_\rho\leq \varepsilon\}$, respectively. We may now state the following extension of Theorem \ref{th:A} for bivariate polynomials.
  \begin{theorem}[Mayer]\label{th:6}
Let $\K$ be a field. Let $f=a_0(x)+a_1(x)y+\cdots+a_n(x) y^n\in \K[x,y]$ be a polynomial with $n\geq 2$, $a_0,a_1,\ldots,a_n\in \K[x]$, $a_0a_n\neq 0$, such that there exists an index $j$ with $0\leq j\leq n$ for which
\begin{eqnarray*}
\|a_j\|_{\rho}>\max_{i\neq j}\|a_i\|_{\rho}.
\end{eqnarray*}
Then $f$ has $j$ zeros in the disk $\|y\|_\rho<1$, and $n-j$ zeros outside the disk  $\|y\|_\rho\leq 1$ in $\overline{\K(x)}$.
\end{theorem}
\begin{proof}[\bf Proof of Theorem \ref{th:6}]
Recall that our absolute value also satisfies the triangle inequality, and note that by using  the condition $\|a_j\|_{\rho}>\max _{i\neq j}\|a_i\|_{\rho}$ in the statement of Theorem \ref{th:6} for a choice of $\rho>n$, we may deduce that
\begin{eqnarray*}
\rho^{\deg a_j}\geq \rho^{1+\max\limits _{i\neq j}\thinspace \deg a_i}&>&n\rho^{\thinspace \max\limits _{i\neq j}\thinspace \deg a_i}\geq \sum_{i=0,~i\neq j}^n \rho^{\deg a_i},
\end{eqnarray*}
which shows that $\|a_j\|_{\rho}$ also satisfies the inequality
\begin{eqnarray*}
\|a_j\|_{\rho}>\sum_{i=0,~i\neq j}^n \|a_i\|_\rho.
\end{eqnarray*}
The proof of Theorem \ref{th:6} can now be easily obtained by following the lines in the original proof of Mayer \cite{Mayer} and replacing the absolute value $|\cdot|$ by the absolute value $\|\cdot\|_{\rho}$, and by also replacing the field of real numbers or the field of complex numbers by the field $\overline{\K(x)}$. The details of the proof will be omitted.

 We will give here an alternative proof of Theorem \ref{th:6} that relies on a Newton polygon argument. So let us construct the Newton polygon $NP(f)$ of the polynomial $f(x,y)=a_0(x)+a_1(x)y+\cdots +a_{n}(x)y^n$ with respect to the valuation $v$ given by $v(a_i)=-\deg a_i$, that is the lower convex hull of the set of points $(i,v(a_i))$ with $i=0,\dots, n$ and $a_i\neq 0$. Note that the extension of $v$ to $\overline{\K(x)}$ will be also denoted by $v$. Since $\deg a_j>\max_{i\neq j}\deg a_i$, we deduce that the point $P_j=(j,-\deg a_j)$ is the only vertex of $NP(f)$ with the smallest $y$-coordinate. This implies that all the edges in $NP(f)$ at the left of $P_j$ (if any) have negative slopes, while all the edges of $NP(f)$ at the right of $P_j$ (if any) have positive slopes. Now recall that by Dumas' Theorem, all the linear factors of $f$ over $\overline{\K(x)}$, that is all the factors of the form $x-\theta$ with $\theta\in \overline{\K(x)}$ a root of $f$, must contribute to $NP(f)$ via translates with their own Newton polygon, that is with an edge having endpoints $(0,v(\theta))$ and $(1,0)$, and hence having slope equal to $-v(\theta)$. Thus, at the left of $P_j$ there are precisely $j$ translates of the edges of width $1$ that have negative slopes (possibly none, if $j=0$), while at the right of $P_j$ there are precisely $n-j$ translates of the edges of width $1$ that have positive slopes (possibly none, if $j=n$). This shows that $f$ has $j$ roots $\theta$ (multiplicities counted) with positive valuation $v(\theta)$ and hence with absolute value $\|\theta\|_{\rho}<1$, and $n-j$ roots $\theta $ (multiplicities counted) with negative valuation $v(\theta)$ and hence with absolute value $\|\theta\|_{\rho}>1$.
\end{proof}
To prove Theorem \ref{th:2}, we will now make use of Theorem \ref{th:6}, as follows.
\begin{proof}[\bf Proof of Theorem \ref{th:2}]
Note that $\|a_i\|_\rho\geq 1$ for all $i=0,\ldots,n-1$, while $\|a_n\|_\rho=1$, since  $a_n\in \K$. Our hypothesis on $\deg a_j$ reads
 \begin{eqnarray*}
\|a_j\|_{\rho}>\max_{i\neq j}\|a_i\|_{\rho},
 \end{eqnarray*}
which shows that $f$ satisfies the hypothesis of Theorem \ref{th:6}. Therefore $f$ has $j$ zeros lying in the disk $\|y\|_\rho< 1$, and the remaining $n-j$ zeros lying outside the disk $\|y\|_\rho\leq 1$ in the algebraic closure $\overline{\K(x)}$.

Suppose now that $f(x,y)=f_1(x,y)f_2(x,y)\cdots f_{r}(x,y)$ is the product of $r$ irreducible polynomials $f_1,\ldots,f_r$ over $\K[x]$, for some integer $r$ with $1\leq r\leq n$. Since $a_0\neq 0$ we must have
\begin{eqnarray*}
1\leq \|a_0\|_\rho=\|f(x,0)\|_\rho=\|f_1(x,0)\|_\rho\cdots\|f_r(x,0)\|_\rho,
\end{eqnarray*}
with $\|f_i(x, 0)\|_\rho\geq 1$ for each $i=1,\ldots,r$. Assume on the contrary that $r>n-j\geq 1$. Then $r>1$ and there exists at least one index $t$ with $1\leq t\leq r$ for which all the zeros of $f_{t}$ lie in the disk $\|y\|_\rho<1$ (for otherwise each of the $r$ factors $f_1,\dots ,f_r$ would have at least one zero outside the disk $\|y\|_\rho\leq 1$, and hence $f$ would be forced to have more than $n-j$ such zeros). Recall now that the leading coefficient of $f$ belongs to $\K$, so the leading coefficient $\alpha_{t}$ of $f_{t}$ must also belong to $\K$, implying that $\|\alpha_t\|_{\rho }=1$. Therefore, if we write  $f_{t}(x,y)=\alpha_{t} \prod_{\theta}(y-\theta)$, with the product running over all the zeros $\theta $ of $f_{t}$ in $\overline{\K(x)}$, we deduce that
\begin{eqnarray*}
1\leq \|f_{t}(x,0)\|_\rho=\|\alpha_{t}\|_\rho\prod_\theta \|\theta\|_\rho<1,
\end{eqnarray*}
which is a contradiction. This completes the proof of the theorem.
\end{proof}
\begin{proof}[\bf Proof of Theorem \ref{th:2withreciprocal}]
Since $a_0\in \K$, the reciprocal $y^nf(x,1/y)$ of $f(x,y)$ with respect to $y$ satisfies the hypotheses in Theorem \ref{th:2} with $n-j$ instead of $j$ (since $y^nf(x,1/y)=b_0(x)+b_1(x)y+\cdots +b_n(x)y^n$ with $b_j=a_{n-j}$ for each $j$), so it is the product of at most $j$ irreducible factors over $\K[x]$. Since the irreducible factors of $y^nf(x,1/y)$ are precisely the reciprocals with respect to $y$ of the irreducible factors of $f$, we conclude that $f$ too is the product of at most $j$ irreducible factors over $\K[x]$.
\end{proof}
\begin{proof}[\bf Proof of Theorem \ref{th:3}]
Let $g(x,y)=a_n^{n-1}f(x,y/a_n)$. We observe that
 \begin{eqnarray*}
g(x,y)=a_0(x)a_n(x)^{n-1}+\cdots+a_j(x)a_n(x)^{n-1-j}y^j+\cdots+a_{n-1}(x)y^{n-1}+y^{n},
 \end{eqnarray*}
which shows that $g\in \K[x][y]$ and $g$ is monic. In particular, the leading coefficient of $g$ belongs to $\K$, and moreover, one may easily check that $f$ and $g$ have the same number of irreducible factors over $\K[x]$, counted with their multiplicities. Thus, it will be sufficient to prove our result for $g$ instead of $f$.
Using our hypothesis on the degree of $a_j$, we observe that the coefficients of $g$ satisfy
 \begin{eqnarray*}
\deg(a_ja_n^{n-1-j})&=&\deg(a_j)+(n-1-j)\deg(a_n)\\
&>&\max_{i\neq j}\thinspace \{\deg(a_i)+(j-i)\deg(a_n)\}+(n-1-j)\deg(a_n)\\
&=&\max_{i\neq j}\thinspace \{\deg(a_i)+(n-1-i)\deg(a_n)\}\\&=&\max_{i\neq j}\thinspace \deg(a_ia_n^{n-1-i}).
 \end{eqnarray*}
Consequently, by Theorem \ref{th:2}, the polynomial $g$ (and hence $f$ too) is a product of at most $n-j$ irreducible polynomials over $\K[x]$, and this completes the proof of the theorem.
\end{proof}
\begin{proof}[\bf Proof of Theorem \ref{th:3withreciprocal}]
By Theorem \ref{th:3} we know that $f$ is a product of at most $n-j$ irreducible polynomials over $\K[x]$, if
\begin{equation}\label{first}
\deg a_j>\max_{i\neq j}\thinspace \{\deg a_i+(j-i)\deg a_n\}.
\end{equation}
Let us consider now the reciprocal of $f$ with respect to $y$, that is
\[
y^nf(x,1/y)=a_{n}(x)+a_{n-1}(x)y+\cdots +a_0(x)y^n.
\]
By Theorem \ref{th:3} with $y^nf(x,1/y)$ instead of $f$ we see that $y^nf(x,1/y)$ (and hence $f$ too) is a product of at most $j$ irreducible polynomials over $\K[x]$, if
\[
\deg a_j>\max_{i\neq n-j}\thinspace \{\deg a_{n-i}+(n-j-i)\deg a_0\},
\]
or equivalently, after relabelling the indices, if
\begin{equation}\label{second}
\deg a_j>\max_{i\neq j}\thinspace \{\deg a_i+(i-j)\deg a_0\}.
\end{equation}
Observe now that conditions (\ref{first}) and (\ref{second}) will both hold if and only if $\deg a_j$ satisfies the pair of inequalities
\begin{eqnarray*}
\deg a_j & > & \max_{i<j}\thinspace \{\deg a_i+(j-i)\deg a_n\} \quad {\rm and}\\
\deg a_j & > & \max_{i>j}\thinspace \{\deg a_i+(i-j)\deg a_0\}.
\end{eqnarray*}
This completes the proof of the theorem.
\end{proof}
\section{Examples}\label{sec:4}
 We will provide here some examples of infinite families of polynomials in $\mathbb{Z}[x,y]$ whose factorization properties can be deduced using the results proved in the preceding section. In all our examples, we will assume that $n\geq 2$.
\begin{example}
For any polynomials $a_1,\dots ,a_n\in\mathbb{Z}[x]$ with $\max\thinspace \{\deg a_1,\dots ,\deg a_n\}\leq 2^n-1$ and $a_n\neq 0$, the polynomial
\begin{eqnarray*}
f_1=1+x^{2^n}+a_1(x)y+\cdots+a_n(x)y^n\in \mathbb{Z}[x,y]
\end{eqnarray*}
satisfies the hypotheses of Corollary \ref{c2} with $a_0=1+x^{2^n}$, which is irreducible, being a cyclotomic polynomial. Therefore $f_1$ must be irreducible over $\mathbb{Z}[x]$.
\end{example}
\begin{example}
For any positive integer $k$, and any nonzero polynomial $b(x)\in\mathbb{Z}[x]$ with $\deg b\leq 2^k+1$ and such that $1+x^{2^k}$ is not a factor of $b$, consider the bivariate polynomial
\begin{eqnarray*}
f_2=(1+x^{2^k})(1+x^2)+b(x)(y+y^2+\cdots+y^{n-1})+(1+x^{2^k})y^n\in \mathbb{Z}[x,y].
\end{eqnarray*}
It is easy to check that $f_2$ satisfies the hypothesis of Corollary \ref{c2} with $a_0=(1+x^{2^k})(1+x^2)$, $q=1+x^{2}$, $a_1=\ldots=a_{n-1}=b$, and $a_n=1+x^{2^k}$, which is irreducible over $\mathbb{Z}$ and satisfies the condition that $\deg a_n\geq \deg a_0-\deg q$. We conclude that $f_2$ is irreducible over $\mathbb{Z}[x]$.
\end{example}
\begin{example}
For any fixed index $j$ with $0\leq j\leq n-1$, and any fixed, arbitrarily chosen nonzero integers $c_0,\dots ,c_n$, the polynomial
\begin{eqnarray*}
f_3=c_0x+c_1xy+\cdots+c_{j-1}xy^{j-1}+c_jx^2y^j+c_{j+1}y^{j+1}+\cdots+c_ny^{n}\in \mathbb{Z}[x,y],
\end{eqnarray*}
satisfies the hypothesis of Theorem \ref{th:2} with
\begin{eqnarray*}
a_i=c_ix\ {\rm for}\ 0\leq i\leq j-1,\ a_j=c_jx^{2},\ {\rm and}\ a_{i}=c_i\ {\rm for}\ j+1\leq i\leq n,
\end{eqnarray*}
so $f_3$ must be a product of at most $n-j$ irreducible polynomials over $\mathbb{Z}[x]$.
\end{example}
\begin{example}
For any fixed index $j$ with $1\leq j\leq n-1$, and any fixed, arbitrarily chosen nonzero integers $c_0,\dots ,c_n$, the polynomial
\begin{eqnarray*}
f_4=c_0+c_1xy+\cdots +c_{j-1}x^{j-1}y^{j-1}+c_{j}x^{j+1}y^j+c_{j+1}xy^{j+1}+\cdots+c_nxy^{n}\in \mathbb{Z}[x,y],
\end{eqnarray*}
satisfies the hypothesis of Theorem \ref{th:3} with
\begin{eqnarray*}
a_i=c_ix^{i}\ {\rm for}\ 0\leq i\leq j-1,\ a_j=c_jx^{j+1},\ {\rm and}\ a_{i}=c_ix\ {\rm for}\ j+1\leq i\leq n,
\end{eqnarray*}
since we have
\begin{eqnarray*}
\deg a_j=j+1>j&=&\deg a_{j-1}+(j-(j-1))\\&=&\max_{0\leq i\leq n,~i\neq j}\{\deg a_i+(j-i)\deg a_n\}.
\end{eqnarray*}
Therefore, $f_4$ is a product of at most $n-j$ irreducible polynomials over $\mathbb{Z}[x]$. On the other hand, since $\deg a_0=0$ and $\deg a_j>\max_{i\neq j}\deg a_i$, we conclude by Theorem \ref{th:3withreciprocal} that actually $f_4$ is a product of at most $\min\{j,n-j\}$ irreducible polynomials over $\mathbb{Z}[x]$.
\end{example}
\begin{example}
For a more general example that shows that the bound in Theorem \ref{th:2withreciprocal} is best possible, let $n$ be a positive integer, $j<n$ a divisor of $n$, and consider the polynomial
\[
f_5(x,y)=(1+xy+y^{\frac{n}{j}})^j\in\mathbb{Z}[x,y],
\]
written also as $f_5=a_0(x)+a_1(x)y+\cdots a_n(x)y^n$, say, with $a_i\in\mathbb{Z}[x]$, and $a_0=a_n=1$. It is easy to see that $f_5$ has a unique coefficient of maximal degree, namely $a_j$, and $\deg a_j=j$. If we ignore the fact that $f_5$ is the $j$th power of $1+xy+y^{\frac{n}{j}}$, and we are only looking at the degrees of its coefficients, we conclude by Theorem \ref{th:2withreciprocal} that $f_5$ is a product of at most $j$ irreducible factors over $\mathbb{Q}[x]$. This is indeed the case, since the polynomial $1+xy+y^{\frac{n}{j}}$ is irreducible over $\mathbb{Q}[x]$. To see this, we will actually prove the more general fact that all the polynomials of the form $1+xy+y^n$ with $n\geq 2$ are irreducible over $\mathbb{Q}[x]$. Denote $1+xy+y^n$ by $g$ and consider the Newton polygon $NP(g)$ of $g$ with respect to the valuation $-\deg(\cdot)$ on $\mathbb{Q}[x]$. It is easy to see that $NP(g)$ has only two segments, the left-one with endpoints $A=(0,0)$ and $B=(1,-1)$, and the one on the right with endpoints $B=(1,-1)$ and $C=(n,0)$. None of these two segments contain lattice points other than their endpoints, so by Dumas' Theorem, $g$ can have at most two irreducible factors, say $h_1$ and $h_2$, with $\deg _yh_1=1$ and $\deg _yh_2=n-1$, say $h_1=b_0(x)+b_1(x)y$ and $h_1=c_0(x)+c_1(x)y+\cdots c_{n-1}(x)y^{n-1}$, and whose Newton polygons $NP(h_1)$ and $NP(h_2)$ should be translates of the two segments $\overline{AB}$ and $\overline{BC}$. On the other hand the equality $g=h_1h_2$ would imply $b_0c_0=1$, so both $b_0$ and $c_0$ should have degree zero, meaning that $NP(h_1)$ should be precisely the segment $\overline{AB}$, thus forcing $b_1$ to have degree $1$. This obviously can not hold, since $b_1c_{n-1}$ should be also equal to $1$. Thus $g$ must be irreducible over $\mathbb{Q}[x]$, as claimed.

\end{example}
\textbf{\large Acknowledgments.}
Ms. Rishu Garg is thankful to the Council of Scientific and Industrial Research (CSIR) for providing her Junior Research Fellowship (JRF) wide grant no. CSIRAWARD/JRF-NET2022/11769 for carrying the present research. This paper is a part of the research work undertaken by her for her Ph.D. thesis.

The present research is also supported by Science and Engineering Research Board (SERB), a
statutory body of Department of Science and Technology (DST), Government of India through the project grant no. MTR/2017/000575 awarded to the last author under the MATRICS Scheme.
\subsection*{Disclosure statement}
The authors report that there are no competing interests to declare.

\end{document}